\documentclass[reqno]{amsart}
\usepackage{lineno,hyperref}
\usepackage{amsfonts, amsmath} 
\usepackage{amssymb}
\usepackage{amsthm}
\usepackage{mathrsfs}
\usepackage[T1]{fontenc}
\usepackage{accents}
\usepackage[mathscr]{euscript}

\usepackage{geometry}
\theoremstyle{definition}

\begin{document}
	\title[\hfilneg Multiple Solutions for Nonlinear Generalized-Kirchhoff Type Potential Systems \hfil]{Multiple Solutions for Nonlinear Generalized-Kirchhoff type potential Systems in Unbounded Domains}

	\author[Nabil Chems Eddine and Anass Ouannasser]
	{Nabil Chems Eddine and Anass Ouannasser} 
	
	\email{nab.chemseddine@gmail.com}
	\email{anass.ouannasser@um5r.ac.ma}
	
	\address{Center of Mathematical Research and Applications of Rabat (CeReMAR), Laboratory of Mathematical Analysis and Applications, Faculty of Sciences, Mohammed V University, P.O. Box 1014, Rabat, Morocco.}

	\keywords{Multiple solutions, variable exponent spaces, Kirchhoff-type problems, $p$-Laplacian, $p(x)$-Laplacian, generalized Capillary operator, critical points theory.}
	\subjclass[2010]{ 35J60, 47J30, 58E05}
	
	\begin{abstract}
		In this paper, we consider a class of quasilinear stationary Kirchhoff type potential systems in unbounded domains, which involves a general variable exponent elliptic operator. Under some suitable conditions on the nonlinearities, we establish existence of at least three weak solutions for the problem. The proof of our main result uses variational methods and the critical theorem of Bonanno and Marano.
	\end{abstract}

	\maketitle
	\numberwithin{equation}{section}
	\newtheorem{theorem}{Theorem}[section]
	\newtheorem{lemma}[theorem]{Lemma}
	\newtheorem{proposition}[theorem]{Proposition}
	\newtheorem{definition}[theorem]{Definition}
	\newtheorem{remark}[theorem]{Remark}
	\allowdisplaybreaks
	
	\section{Introduction}
	
	The aim of this paper is to show the existence of at least three weak solutions for the following class of nonlocal quasilinear elliptic systems in $\mathbb{R}^N$
	
	\begin{eqnarray}
		\label{s1.1}
		\begin{cases}
			-M_i\Big(\mathcal{B}_i(u_i)\Big)
			\textrm{div}\,\Big(  \mathcal{A}_{1i}(\nabla u_i)+ w_i(x)\mathcal{A}_{2i}(u_i)\Big)=\lambda F_{u_i}(x,u) & \text{in }\mathbb{R}^N, \\
			\qquad u_i \in W_{w_i}^{1,p_i(x)}(\mathbb{R}^N)\cap W_{w_i}^{1,\gamma_i(x)}(\mathbb{R}^N) ;
		\end{cases}
	\end{eqnarray}
	for $1\leq i \leq n$  ($n\in \mathbb{N} $), where $N\geq 2$,  $\lambda$ is a positive real parameter and the functions $w_i \in L^{\infty}(\mathbb{R}^N)$ such that $\displaystyle w_i := ess \inf_{x\in \mathbb{R}^N} w_i(x) > 0$. $M_i$ are bounded continuous   functions, $F$ belongs to $C^1(\mathbb{R}^{N}\times \mathbb{R}^{n})$ and satisfies adequate growth assumptions and   $F_{u_i}$ denotes the partial derivative of $F$ with respect to $u_i$ for all $1\leq i \leq n$. The functions $p_i(x), q_i(x)$ and $\gamma_i(x)$ are  continuous real-valued functions such that
	\begin{gather}
		1 <p_i^{-}\leq p_i(x) \leq p_i^{+}< q_i^{-}\leq q_i(x) \leq q_i^{+}
		< N,
	\end{gather}
	and
	$$\gamma_i(x) =(1-\mathcal{H}(k^3_{i})) p_i(x)+ \mathcal{H}(k^3_{i})q_i(x),$$
	
	for all $x\in \mathbb{R}^N$ with $k^3_{i}$ is given in $(\textbf{\textit{H}}_2)$, where $\displaystyle p_i^-:= \inf_{x\in \mathbb{R}^N}p_i(x)$,  $\displaystyle p_i^{+}:= \sup_{x\in \mathbb{R}^N}p_i(x)$, and analogously to $q_i^-, q_i^+, \gamma_i^-$ and $\gamma_i^+$, with $\displaystyle \mathcal{H}:\mathbb{R}_{0}^{+}\to \left\lbrace 0,1\right\rbrace  $ is given by
	\[
	\mathcal{H}(k_i)=\begin{cases}
		1 & \text{ if }        	k_i \geq 0,\\
		0 & \text{ if }            k_i<0.
	\end{cases}
	\]
	
	The operator $\mathcal{A}_{ji} : X_i\to \mathbb{R}^n$ with $j=1$ or $2$, and the operator  $ \mathcal{B}_i : X_i\to \mathbb{R}$, are respectively defined by  
	\begin{equation}\label{AB}
		\mathcal{A}_{ji}(u_i)= a_{ji}(| u_i|^{p_i(x)}) | u_i|^{p_i(x)-2} u_i, \text{ and }\mathcal{B}_i(u_i)= \displaystyle\int_{\mathbb{R}^N }\dfrac{1}{p_i(x)}\Big(A_{1i}(|\nabla u_i|^{p_i(x)})+ w_i(x)A_{2i}(|u_i|^{p_i(x)})\Big)dx,
	\end{equation}
	where $X_i$ is the Banach space $X_i:= W_{w_i}^{1,p_i(x)}(\mathbb{R}^N)\cap W_{w_i}^{1,\gamma_i(x)}(\mathbb{R}^N)$, $A_{ji}(.)$ is the function $A_{ji}(t)=\displaystyle\int_{0}^{t}a_{ji}(k) dk$, and the function $a_i(.)$ is described in the hypothesis $(\textbf{\textit{H}}_1)$. \\
	
	In this article, we consider the function $ a_{ji}: \mathbb{R}^+ \to \mathbb{R}^+$  satisfying the following hypotheses for all $1\leq i \leq n$ and $j\in \{1,2\}$:
	
	\begin{itemize}
		\item[$(\textbf{\textit{H}}_1)$] The function $a_{ji}(.)$ is of   class $C^1$.
		
		\item[$(\textbf{\textit{H}}_2)$] There exist positive constants $k_{ji}^0, k_{ji}^1, k_{ji}^2$ and $k_{i}^3$ for all $1\leq i\leq n$, such that
		
		$$k_{ji}^0 + \mathcal{H}(k_{i}^3)k_{ji}^2 \tau^{\frac{q_i(x)-p_i(x)}{p_i(x)}} \leq a_{ji}(\tau ) \leq k_{ji}^1 + k_{i}^3\tau^{\frac{q_i(x)-p_i(x)}{p_i(x)}},
		$$
		for all $\tau \geq 0$ and for almost every $x\in \mathbb{R}^N.$
		
		\item[$(\textbf{\textit{H}}_3)$] There exist $c_{ji}>0$ such that
		
		$$ \min \left\lbrace a_{ji}(\tau^{p_i(x)})\tau^{p_i(x)-2}, a_{ji}(\tau^{p_i(x)})\tau^{p_i(x)-2}
		+\tau \frac{\partial(a_{ji}(\tau^{p_i(x)})\tau^{p_i(x)-2})}{\partial \tau } \right\rbrace \geq c_{ji}\tau^{p_i(x)-2},
		$$
		for almost every $x\in \mathbb{R}^N$ and for all $\tau>0$.
		
		\item[$(\textbf{\textit{H}}_4)$]
		There exists positive constants $\beta_{ji}$ for all $i\in \left\lbrace 1,2,\dots,n\right\rbrace $ such that
		
		$$ A_{ji}(\tau )\geq \frac{1}{\beta_{ji}}a_{ji}(\tau)\tau,$$
		for all $\tau \geq 0$.
		
		\item[$(\textbf{\textit{M}})$] $M_i: \mathbb{R}^+ \longrightarrow \mathbb{R}$ are continuous and increasing functions such that $0 < m_0=\displaystyle\min_{1 \leq i \leq n} M_i(0)  \leq M_i(t) \leq m_1=\max_{1 \leq i \leq n}M_i(t)$, for all $t \geq 0$, $i\in \left\lbrace 1,2,\dots,n\right\rbrace $. \\
	\end{itemize}

     Recently, a great attention has been focused on the study of this type of problems. They appear in mathematical models in different branches in science as electrorheological fluids \cite{Ru1,Ru2}, elastic mechanics \cite{Zhik1}, stationary thermorheological viscous flows of non-Newtonian fluids \cite{Raj}, image processing \cite{Chen1}, and mathematical description of the processes filtration of barotropic gas through a porous medium \cite{Anto2}. \\
     	
     The system (\ref{s1.1}) is related (in the case of a single equation)  to a model firstly proposed by Kirchhoff in 1883 as the stationary version of the Kirchhoff equation
     
     \begin{equation}
     	\label{e1.1}
     	\rho\frac{\partial^2u}{\partial t^2}-\left( \dfrac{\rho_0}{h}+\dfrac{E}{2L}\int_{0}^{L}\left\vert\dfrac{\partial u(x)}{\partial x}\right\vert^2dx \right)\dfrac{\partial^2u}{\partial x^2}=0,
     \end{equation}
     where $\rho$, $\rho_0$, $E$ and $L$ are constants. This equation   extends the classical D'Alembert's wave equation by considering the effects of the changes in the length of the strings during the vibrations. A distinguishing feature of equation (\ref{e1.1}) is that the equation contains a nonlocal coefficient $\dfrac{\rho_0}{h}+\dfrac{E}{2L}{\displaystyle\int_{0}^{L}\left\vert\dfrac{\partial u}{\partial x}\right\vert^2dx}$ which depends on the average $\dfrac{1}{L}{\displaystyle\int_{0}^{L}\left\vert\dfrac{\partial u}{\partial x}\right\vert^2dx}$, and hence the equation is no longer a pointwise equation. The parameters in equation (\ref{e1.1}) have the following meanings: $E$ is the Young modulus of the material, $\rho$ is the mass density, $L$ is the length of the string, $h$ is the area of cross-section, and $\rho_0$ is the initial tension (see \cite{Kir}). \\
	
	Now, in order to illustrate the degree of generality of the kind of problems studied here, with adequate hypotheses on the functions $a_{ji}$, in the following we present more some examples of problems which are also interesting from the mathematical point of view and have a wide range of applications in physics and related sciences. \\
	
	\textbf{Example I.} Considering $a_{ji}\equiv 1$, we have that $a_{ji}$ satisfies the  $(\textbf{\textit{H}}_1),(\textbf{\textit{H}}_2)$ and $(\textbf{\textit{H}}_3)$ with $k^0_{ji}=k^1_{ji}=1$  and $k^2_{ji}>0$ and $k^3_{i}=0$ for all $i \in \{1,2,...,N\}$ and for $j=1$ or $2$. In this case we are studying problem:
	
	\begin{eqnarray}
		\label{ex1.1}
		\begin{cases}
			-M_i\Big(\mathcal{B}_i(u_i)\Big)\Big(\Delta_{p_i(x)}u_i -w_i(x)| u_i|^{p_i(x)-2} u_i\Big)
			= \lambda F_{u_i}(x,u)  & \text{ in } \mathbb{R}^N, \\
			\quad	u_i \in W_{w_i}^{1,p_i(x)}(\mathbb{R}^N)\\
			
		\end{cases}
	\end{eqnarray}
	
	where $\mathcal{B}_i(u_i)=\displaystyle\int_{\mathbb{R}^N }\dfrac{1}{p_i(x)}\left(|\nabla u_i|^{p_i(x)}+ w_i(x)|u_i|^{p_i(x)} \right)dx$. The operator $ \Delta_{p_i(x)}u_i:=\textrm{div}\,(|\nabla u_i|^{p_i(x)-2}\nabla u_i)$ is so-called $p_i(x)$-Laplacian, which coincides with the usual $p_i$-Laplacian when $p_i(x)=p_i$, and with the Laplacian when $p_i(x)=2$. \\
	
	\textbf{Example II.} Considering $a_{ji}(t)= 1+ t^{\frac{q_i(x)-p_i(x)}{p_i(x)}}$, 
	we have that $a_{ji}$ satisfies the  $(\textbf{\textit{H}}_1),(\textbf{\textit{H}}_2)$ and $(\textbf{\textit{H}}_3)$ with $k^0_{ji}=k^1_{ji}=k^2_{ji}=\kappa_i^3=1$ for all $i\in \{1,2,...,n\}$ and for $j=1$ or $2$). In this case we are studying the $p_i\&q_i$-Laplacian equation:
	\begin{eqnarray}
		\label{ex1.12}
		\begin{cases}
			-M_i\Big(\mathcal{B}_i(u_i)\Big)\Big(\Delta_{p_i(x)}u_i +\Delta_{q_i(x)}u_i -w_i(x) (| u_i|^{p_i(x)-2} u_i +| u_i|^{q_i(x)-2} u_i) \Big)
			= \lambda F_{u_i}(x,u)  & \text{ in } \mathbb{R}^N, \\
			\qquad	u_i \in W_{w_i}^{1,p_i(x)}(\mathbb{R}^N)\cap W_{w_i}^{1,q_i(x)}(\mathbb{R}^N)\\
		\end{cases}
	\end{eqnarray}
	
	where $\mathcal{B}_i(u_i)=\displaystyle\int_{\mathbb{R}^N }\Bigg(\dfrac{1}{p_i(x)}\left(|\nabla u_i|^{p_i(x)}+ w_i(x)|u_i|^{p_i(x)} \right)+\dfrac{1}{q_i(x)}\left(|\nabla u_i|^{q_i(x)}+ w_i(x)|u_i|^{q_i(x)} \right)\Bigg)dx$. 
	
	This class of problems comes, for example, from a general reaction-diffusion system
	
	\begin{equation} \label{DCE}
		u_t= \textrm{div}[D(u)\nabla u ]+ h(x,u),
	\end{equation}
	
	where $D(u)=|\nabla u|^{p(x)-2}+|\nabla u|^{q(x)-2}$, and the reaction term $h(x, u)$ is a polynomial of u with variable coefficients. This system has a wide range of applications in physics and related sciences, such as biophysics, plasma physics and chemical reaction design. In such applications, the function $u$ describes a concentration, the first term on the right-hand side of (\ref{DCE}) corresponds to the diffusion with a diffusion coefficient $D(u)$; whereas the second one is the reaction and relates to source and loss processes. Typically, in chemical and biological applications, (for further details, see \cite{Mahshid,He} references therein). \\
	
	We continued with other examples that are also interesting from mathematical point of view: \\
	
	\textbf{Example III.} Considering $a_{1i}(t)= 1+ \frac{t}{\sqrt{1+t^2}}$ and  $a_{2i}\equiv 1$, we have that $a_{1i}$ and $a_{2i}$ satisfies the  $(\textbf{\textit{H}}_1),(\textbf{\textit{H}}_2)$ and $(\textbf{\textit{H}}_3)$ with $k^0_{1i}=k^0_{2i}=k^1_{2i}=1$,$k^1_{1i}=2$, and $k^3_{i}=0$, $k^2_{1i}>0$ and $k^2_{2i}>0$. In this case we are studying problem:

	\begin{eqnarray}
		\label{ex1.13}
		\begin{cases}
			-M_i\Big(\mathcal{B}_i(u_i)\Big)\Big(\textrm{div}\,\Big( \Big(1+ \frac{|\nabla u_i|^{p_i(x)}}{\sqrt{1+|\nabla u_i|^{2p_i(x)}}}\Big)|\nabla u_i|^{p_i(x)-2}\nabla u_i\Big) -w_i(x)| u_i|^{p_i(x)-2} u_i\Big)
			= \lambda F_{u_i}(x,u)   & \text{ in } \mathbb{R}^N, \\
			\quad	u_i \in W_{w_i}^{1,p_i(x)}(\mathbb{R}^N),
		\end{cases}
	\end{eqnarray}
	
	where $\mathcal{B}_i(u_i)=\displaystyle\int_{\mathbb{R}^N }\frac{1}{p_i(x)}\Big(|\nabla u_i|^{p_i(x)} + \sqrt{1+|\nabla u_i|^{2p_i(x)}}+ w_i(x)|u_i|^{p_i(x)}\Big)dx$. \\
	
	The operator  $\textrm{div}\,\Big( \Big(1+ \frac{|\nabla u|^{p(x)}}{\sqrt{1+|\nabla u|^{2p(x)}}}\Big)|\nabla u|^{p(x)-2}\nabla u\Big)$ is so-called  $p(x)$-Laplacian like or so-called the generalized Capillary operator. The Capillarity can be briefly explained by considering the effects of two opposing forces: adhesion, i.e. the attractive (or repulsive) force between the molecules of the liquid and those of the container; and cohesion, i.e. the attractive force between the molecules of the liquid. The study of capillary phenomenon has gained much attention. This increasing interest is motivated not only by the fascination in naturally-occurring phenomena such as motion of drops, bubbles and waves, but also its importance in applied fields ranging from industrial and biomedical and pharmaceutical to microfluidic systems see \cite{Ni}. \\
	
	\textbf{Example IV.} Considering $\displaystyle a_{1i}(t)= 1+ t^{\frac{q_i(x)-p_i(x)}{p_i(x)}} +\frac{1}{(1+t)^{\frac{p_i(x)-2}{p_i(x)}}}$ and $\displaystyle a_{2i}(t)= 1+ t^{\frac{q_i(x)-p_i(x)}{p_i(x)}}$, we have that $\displaystyle a_{1i}$ and $\displaystyle a_{2i}$  satisfies the  $\displaystyle (\textbf{\textit{H}}_1),(\textbf{\textit{H}}_2)$ and $(\textbf{\textit{H}}_3)$ with $k^0_{1i}=k^0_{2i}=k^1_{2i}=1$, $k^1_{1i}= 2$ and $k^3_{i}=k^2_{1i}=k^2_{2i}=1$. In this case we are studying problem

	\begingroup\makeatletter\def \f@size{8}\check@mathfonts
	\begin{eqnarray}
		\label{ex1.13}
		\begin{cases}
			-M_i\Big(\mathcal{B}_i(u_i)\Big)\Bigg[\Delta_{p_i(x)}u_i +\Delta_{q_i(x)}u_i +\textrm{div}\,\Big( \dfrac{|\nabla u_i|^{p_i(x)-2}\nabla u_i}{(1+ |\nabla u_i|^{p_i(x)})^{\frac{p_i(x)-2}{p_i(x)}}}\Big)-w_i(x) \Big(| u_i|^{p_i(x)-2} u_i +| u_i|^{q_i(x)-2} u_i \Big)\Bigg]
			= \lambda F_{u_i}(x,u)  & \text{ in } \mathbb{R}^N, \\
			\qquad	u_i \in W_{w_i}^{1,p_i(x)}(\mathbb{R}^N)\cap W_{w_i}^{1,q_i(x)}(\mathbb{R}^N),
		\end{cases}
	\end{eqnarray}
	\endgroup
	
	where $\mathcal{B}_i(u_i)=\displaystyle\int_{\mathbb{R}^N }\Bigg(\dfrac{1}{p_i(x)}\left(|\nabla u_i|^{p_i(x)}+ w_i(x)|u_i|^{p_i(x)} \right)+\dfrac{1}{q_i(x)}\left(|\nabla u_i|^{q_i(x)}+ w_i(x)|u_i|^{q_i(x)} \right) + \frac{1}{2}(1+ |\nabla u_i|^{p_i(x)})^{\frac{2}{p_i(x)}}\Bigg)dx$. \\

	In the literature, the existence and multiplicity of the solutions for quasilinear elliptic systems have been studied by many authors (\cite{Afrouzi,BonaCand1,Bona1,Bona2,Chems0,Dai1,Dai2,Liu1, Shi1}), where the general variable exponent elliptic operators $\mathcal{A}_{ji}$ in divergence form and the nonlinear potential $F$ have different mixed growth conditions. For example in \cite{Djellit} the authors show the existence of nontrivial solutions for the following $(p,q)$-Laplacian system
	\begin{eqnarray}
		\label{ss1.2}
		\begin{cases}
			\displaystyle -\Delta _{p}u=\frac{\partial F}{\partial u}(x,u,v)\quad \text{in }\mathbb{R}^N \\
			& \\
			\displaystyle -\Delta _{q}v=\frac{\partial F}{\partial v}(x,u,v)\quad \text{in }\mathbb{R}^N
		\end{cases}
	\end{eqnarray}
	where the potential function $F$ satisfies mixed and subcritical growth conditions and, in addition, is supposed to be intimately connected with the first eigenvalue of the $(\Delta _{p},\Delta _{q})$-operator. They apply the Mountain Pass theorem to get existence of nontrivial solutions.  In \cite{Chems}, the authors using an abstract
	critical point result of Bonanno and Marano established the existence of an interval $\Lambda \subseteq [0, +\infty[$ such that for each $\lambda \in \Lambda$ the quasilinear elliptic system
	\begin{equation}
		\label{sSS1.1}
		-\Delta _{p_i(x)}u_i +a_i(x)|u_i|^{p_i(x)-2}u_i=\lambda F_{u_i}(x,u_1,u_2,...,u_n) \quad\text{in }\mathbb{R}^N,
	\end{equation}
	for $1\leq i \leq n$, where $a_i \in L^{\infty}(\mathbb{R}^N)$ such that $\displaystyle a_i := ess \inf_{x\in \mathbb{R}^N} a_i(x) > 0$ and  $\lambda $ is a positive parameter, has at least three distinct nontrivial solutions, and in \cite{Chems1} some similar results were obtained for the following classe of nonlocal elliptic systems
	 	\begin{eqnarray}
		\label{ssss1.1}
		\begin{cases}
			-M_1\left({\displaystyle\int_{\mathbb{R}^N }\dfrac{1}{p(x)}\left(|\nabla u|^{p(x)}+ a(x)|u|^{p(x)}\right)dx}\right)
			\big(\Delta _{p(x)}u -a(x)|u|^{p(x)-2}u\big)=\lambda F_{u}(x,u,v) &\quad\text{in }\mathbb{R}^N, \\
			
			-M_2\left({\displaystyle\int_{\mathbb{R}^N }\dfrac{1}{q(x)}\left(|\nabla v|^{q(x)}+ a(x)|v|^{q(x)}\right)dx}\right)
			\big(\Delta _{q(x)}v -b(x)|v|^{q(x)-2}v\big)=\lambda F_{v}(x,u,v) &\quad\text{in }\mathbb{R}^N;
		\end{cases}
	\end{eqnarray}
	
	where $a $, $b \in L^{\infty}(\mathbb{R}^N)$ such that $\displaystyle a := ess \inf_{x\in \mathbb{R}^N} a(x) > 0$ and $\displaystyle b := ess \inf_{x\in \mathbb{R}^N} b(x) > 0$. $M_1$ and $M_2$ are bounded continuous   functions,   $F$ belongs to $C^1(\mathbb{R}^{N}\times \mathbb{R}^{2})$ and verifies some mixed growth conditions.\\

    This paper's primary purpose is to establish the existence of some interval which includes $\lambda$, where the system (\ref{s1.1}) admits at least three weak solutions, which extend, complement and complete in several ways some of many works in particular the results in \cite{Chems,Chems1}, by means of an abstract critical points result of G. Bonanno  and S.A. Marano \cite{BonaMara}, which is a more precise version of Theorem 3.2 of \cite{BonaCand1}. For other basic notations and definitions we refer to \cite{Zeidler}. \\
	
	Now, and for the convenience of the readers, we recall the three critical points theorem of G. Bonanno and S.A. Marano \cite{BonaMara} which is our main tool to prove the results. Here, $X^*$ denotes the dual space of $X$. 
	
	\begin{lemma}[see {\cite[Theorem 3.6]{BonaMara}}]\label{lem1.1}
		Let $X$ be a reflexive real Banach space; $\Phi:X\rightarrow \mathbb{R}$ be a 
		coercive, continuously G\^{a}teaux differentiable and sequentially weakly 
		lower semicontinuous functional whose G\^{a}teaux derivative admits a 
		continuous inverse on $X^*$; $\Psi:X\rightarrow \mathbb{R}$ be a continuously G\^{a}teaux differentiable functional whose G\^{a}teaux derivative is 
		compact such that
		$$
		\Phi(0)=\Psi(0)=0.
		$$
		Assume that there exist $r>0$ and $\overline{x}\in X,$ with $r<\Phi
		(\overline{x}),$ such that
		\begin{enumerate}
			\item[$\rm(a_1)$] $\displaystyle \frac{\displaystyle \sup_{\Phi(\bar x)\leq r}\Psi(\bar x)}{r}<\frac{\Psi 
				(\overline{x})}{\Phi(\overline{x})}$;
			\item[$\rm(a_2)$] for each $\displaystyle \lambda\in \Lambda_r:=]\frac{\Phi (\overline{x})}{\Psi(\overline{x})},\frac{r}{\displaystyle \sup_{\Phi(\bar x)\leq r}\Psi(\bar x)}[$ the 
			functional $\Phi-\lambda\Psi$ is coercive.
		\end{enumerate}
		Then, for each $\lambda\in\Lambda_r$ the functional $\Phi-\lambda
		\Psi$ has at least three distinct critical points in $X$.
	\end{lemma}

	\textit{Organization of the paper.} The rest of the paper is organized as follows: Section 2 contains some basic preliminary knowledge of the variable exponent spaces and some results which will be needed later. Finally, in Section 3, we state and establish our main result.

	\section{Preliminaries and basic notations}\label{sec2}
	In this section, we introduce some definitions and results which will be used in the next section. Firstly, we introduce some theories of Lebesgue-Sobolev spaces with variable exponent. The details can be found in \cite{Dien1, Fan3, Kov1, Rad}. Denote by $S(\mathbb{R}^N)$ the set of all measurable real functions on $\mathbb{R}^N$. Set 
	
	\begin{equation*}
		C_{+}(\mathbb{R}^N)= \{  p \in C(\mathbb{R}^N) :  \inf_{x\in \mathbb{R}^N}p(x)> 1 \} .
	\end{equation*}
	
	For any $ p \in C_{+}(\mathbb{R}^N)$ we define
	\begin{gather*}
		p^{-}:= \inf_{x\in \mathbb{R}^N}p(x) \quad\text{ and} ~p^{+}:= \sup_{x\in \mathbb{R}^N}p(x) .
	\end{gather*}
	
	For any $p \in C_{+}(\mathbb{R}^N)$, we define the variable exponent Lebesgue space as
	
	\begin{gather*}
		L^{p(x)}(\mathbb{R}^N)= \left\{ u \in S(\mathbb{R}^N) : \int_{\mathbb{R}^N}|u(x)|^{p(x)}dx <\infty \right\},
	\end{gather*}
	
	endowed with the Luxemburg norm
	
	\begin{equation*}
		|u|_{p(x)}:=|u|_{L^{p(x)}(\mathbb{R}^N)} = \inf \left\{ \mu >0 : \int_{\mathbb{R}^N }\left|\frac{u(x)}{\mu}\right|^{p(x)} dx \leq 1 \right \}.
	\end{equation*}
	
	Let $w \in S(\mathbb{R}^N)$, and $w(x)>0$ for a.e $x\in \mathbb{R}^N$. Define the weighted variable exponent Lebesgue space $L_{w}^{p(x)}(\mathbb{R}^N)$ by
	
	\begin{gather*}
		L_{w}^{p(x)}(\mathbb{R}^N)= \left\{ u \in S(\mathbb{R}^N) : \int_{\mathbb{R}^N}w(x)|u(x)|^{p(x)}dx <\infty \right\},
	\end{gather*}
	
	with the norm
	
	\begin{equation*}
		|u|_{p(x),w(x)}:=|u|_{L_w^{p(x)}(\mathbb{R}^N)} = \inf \left\{ \mu >0 : \int_{\mathbb{R}^N }w(x)\left|\frac{u(x)}{\mu}\right|^{p(x)} dx \leq 1 \right \}.
	\end{equation*}
	
	From now on, we suppose that $w \in L^{\infty}(\mathbb{R}^N)$ with $\displaystyle w := ess \inf_{x\in \mathbb{R}^N} w(x) > 0$. Then obviously $L_w^{p(x)}(\Omega)$ is a Banach space (see \cite{Cruz1} for details). \\
	
	On the other hand, the variable exponent Sobolev space $W^{1,p(x)}(\mathbb{R}^N)$ is defined by
	
	\[
	W^{1,p(x)}(\mathbb{R}^N)=\{ u\in L^{p(x)}(\mathbb{R}^N):| \nabla
	u| \in L^{p(x)}(\mathbb{R}^N)\},
	\]
	
	and is endowed with the norm
	
	\[
	\| u\| _{1,p(x)}:=\| u\|_{W^{1,p(x)}(\mathbb{R}^N)}
	=|u| _{p(x)}+| \nabla u| _{p(x)},   ~~~~ \forall u\in W^{1,p(x)}(\mathbb{R}^N ).
	\]
	
	Next, the weighted-variable exponent Sobolev space $W_{w}^{1,p(x)}(\mathbb{R}^N)$ is defined by
	
	\[
	W_{w}^{1,p(x)}(\mathbb{R}^N)=\{ u\in L_{w}^{p(x)}(\mathbb{R}^N):| \nabla
	u| \in L_{w}^{p(x)}(\mathbb{R}^N)\},
	\]
	
	with the norm
	
	\begin{equation*}
		\| u\| _{p(x),w}:= \inf \left\{ \mu >0 :  \int_{\mathbb{R}^N } \left|\frac{\nabla u(x)}{\mu}\right|^{p(x)} +w(x)\left|\frac{u(x)}{\mu}\right|^{p(x)} dx \leq 1 \right \},  \forall u \in W_{w}^{1,p(x)}(\mathbb{R}^N).
	\end{equation*}
	
	Then the norms $\| u\| _{p(x),w}$ and $\| u\| _{p(x)}$ are equivalent in $W_{w}^{1,p(x)}(\mathbb{R}^N)$.
	
	If $p^->1$, then the spaces $L^{p(x)}(\mathbb{R}^N)$, $W^{1,p(x)}(\mathbb{R}^N)$ and $W_w^{1,p(x)}(\mathbb{R}^N)$ are separable, reflexive and uniformly convex Banach spaces. \\
	
	Here we display some facts which will be used later.
	
	\begin{proposition}[see \cite{Dien1, Fan3}] \label{prop1}
			The topological dual space of
		$L^{p(x)}(\Omega)$ is $L^{p'(x)}(\Omega)$, where
		\[
		\frac{1}{p(x)}+\frac{1}{p'(x)}=1.
		\]
		Moreover, for any $(u,v)\in L^{p(x)}(\Omega)\times L^{p'(x)}(\Omega)$,
		we have
		\[
		\Big| \int_{\Omega}uvdx\Big|
		\leq (\frac{1}{p^{-}}+ \frac{1}{(p')^{-}})| u|_{p(x)}| v| _{p'(x)}
		\leq 2| u| _{p(x)}| v|_{p'(x)}.
		\]
	\end{proposition}
	
	\begin{proposition}[see \cite{Dien1, Fan3}] \label{prop2}
		Denote $\displaystyle \rho_{p(x)} (u):=\int_{\mathbb{R}^N}| u| ^{p(x)}dx$,  for all $u\in L^{p(x)}(\mathbb{R}^N)$. We have
		\[
		\min \{ | u| _{p(x)}^{p^{-}},| u| _{p(x)}^{p^{+}}\}
		\leq \rho_{p(x)} (u)\leq \max \{ | u| _{p(x)}^{p^{-}},| u| _{p(x)
		}^{p^{+}}\},
		\]
		
		and the following implications are true  
		\begin{itemize}
			\item[(i)]  $|u| _{p(x)}<1$ (resp. $=1, >1$) $\Leftrightarrow \rho_{p(x)} (u)<1$
			(resp. $=1,>1$),
			
			\item[(ii)] $|u| _{p(x)}>1 \Rightarrow | u| _{p(x)}^{p^{-}}\leq \rho_{p(x)} (u)
			\leq | u| _{p(x)}^{p^{+}}$,
			
			\item[(iii)] $|u|_{p(x) }<1\Rightarrow | u| _{p(x)}^{p^{+}}\leq \rho_{p(x)} (u)
			\leq | u| _{p(x)}^{p^{-}}$.		
		\end{itemize}
	\end{proposition}

Denote $\displaystyle \rho_{p(x),w}(u):=\int_{\mathbb{R}^N }\left( \left|\nabla u(x)\right|^{p(x)} +w(x)\left|u(x)\right|^{p(x)} \right)dx$,  for all $u \in W_{w}^{1,p(x)}(\mathbb{R}^N)$. From Proposition \eqref{prop2}, we have
		
		\begin{equation}
			\label{inq1}
			\| u\| _{p(x),w}^{p^{-}}  \leq \rho_{p(x),w}(u) \leq  \| u\| _{p(x),w}^{p^{+}}    \quad \text{ if} ~~~~ \| u\| _{p(x),w} \geq 1,
		\end{equation}
		\begin{equation}
			\label{inq2}
			\| u\| _{p(x),w}^{p^{+}}  \leq \rho_{p(x),w}(u) \leq  \| u\| _{p(x),w}^{p^{-}}    \quad \text{ if} ~~~~ \| u\| _{p(x),w} \leq 1.
		\end{equation}

	\begin{proposition}[see \cite{Edmu1}] \label{prop3}
		Let $p(x)$ and $q(x)$ be measurable functions such that $p\in L^{\infty }(\mathbb{R}^N)$ and $1\leq p(x), q(x)< \infty $ almost everywhere in
		$\mathbb{R}^N$. If $u\in L^{q(x)}(\mathbb{R}^N)$, $u\neq 0$, then we have
		
		\begin{gather*}
			| u| _{p(x)q(x)}\leq 1\Rightarrow |u| _{p(x)q(x)}^{p^{-}}
			\leq \big|| u| ^{p(x)}\big| _{q(x)}\leq | u| _{p(x)q(x)}^{p^{+}},
			\\
			| u| _{p(x)q(x)}\geq 1\Rightarrow |u| _{p(x)q(x)}^{p^{+}}
			\leq \big| | u| ^{p(x)}\big| _{q(x)}\leq | u| _{p(x)q(x)}^{p^{-}}.
		\end{gather*}
		
		In particular, if $p(x)=p$ is constant, then
		
		\[
		| | u| ^{p}| _{q(x)}
		=|u| _{pq(x)}^{p}.
		\]
	\end{proposition}

	Now, for all $x\in \mathbb{R}^N$, denote 
	\[
	p^{\partial }(x)=\begin{cases}
		\frac{N p(x)}{N-p(x)} &\text{for } p(x)<N \\
		+\infty &\text{for }p(x)\geq N
	\end{cases}
	\]
	the critical Sobolev exponent of $p(x)$.  We have

	\begin{proposition}[see \cite{Dien1,Edmu1}] \label{prop4}
		Let $r\in C_{+}^{0,1}(\mathbb{R}^N)$, the space of Lipschitz-continuous
		functions defined on $\mathbb{R}^N$. Then, there exists a positive constant $c_{r}$ depending on $r$ such that
		
		\[
		| u| _{r^{\partial }(x)}\leq c_{r}\|u\|_{r(x),w},   \quad \forall
		u\in W_w^{1,r(x)}(\mathbb{R}^N).
		\]
	\end{proposition}
	
	\begin{proposition}[see \cite{Dien1,Edmu1}] \label{prop5}
		Assume that $p\in  C(\mathbb{R}^N) $ with $p(x)>1$ for each $x\in \mathbb{R}^N$. If $r\in C(\mathbb{R}^N)$ is such that $1<r(x)<p^{\partial}(x)$ for each $x\in \Omega$, then there exists a continuous and compact embedding $W^{1,p(x)}(\mathbb{R}^N)\hookrightarrow L^{r(x)}(\mathbb{R}^N )$.
	\end{proposition}
	
	In the following, we shall use the product space
	
	\[
	X:=\prod_{i=1}^{n}\Big(W_{w_i}^{1,p_i(x)}(\mathbb{R}^N )\cap W_{w_i}^{1,\gamma_i(x)}(\mathbb{R}^N )\Big),
	\]
	
	equipped with the norm
	
	\[
	\| u\|:=\max \left\lbrace  \|  u_i\|_{w_i} \right\rbrace ,~~~~~ \forall u=(u_1, \dots, u_n)\in X,
	\]
	
	where $\| u_i\|_{w_i}:= \|  u_i\|_{p_i(x),w_i}+ \mathcal{H}(\kappa_i^3)\| u_i\|_{q_i(x),w_i}$  is the norm in $W_{w_i}^{1,p_i(x)}(\mathbb{R}^N )\cap W_{w_i}^{1,\gamma_i(x)}(\mathbb{R}^N )$. We denote  $X^{\star}$ the topological dual of $X$ equipped with the usual dual norm.\par
	
	\begin{definition}
		Let $X$ be a Banach space, an element $u=(u_1,u_2,...,u_n) \in X$ is called a weak solution of the system \eqref{s1.1} if
		
		\begin{center}
			$\displaystyle \sum_{i=1}^{n} M_i\left(\mathcal{B}_i(u_i)\right)\int_{\mathbb{R}^N} \Big(\mathcal{A}_{1i}(\nabla u_i) \nabla v_i + w_i(x)\mathcal{A}_{2i}(u_i) v_i \Big)\,dx - \sum_{i=1}^{n}\displaystyle\int_{\Omega} \lambda F_{u_i}(x,u_1,...u_n)v_i\,dx =0,$
		\end{center}
		
		for all $\displaystyle v=(v_1,v_2,...,v_n)\in X= \prod_{i=1}^{n}(W_{w_i}^{1,p_i(x)}(\mathbb{R}^N )\cap W_{w_i}^{1,\gamma_i(x)}(\mathbb{R}^N ))$.
	\end{definition}
	
	For every $u = (u_1, \dots, u_n)$ in $X$, let us define the functional $\Psi$ by
	
	\begin{center}
		$\displaystyle \Psi(u) = \int_{\mathbb{R}^N} F(x, u_1, \dots, u_n) dx$.
	\end{center}
	
	Under the assumptions below $(\mathcal{F}_1)$ and $(\mathcal{F}_2)$, we have $\displaystyle \Psi \in C^1(X, \mathbb{R})$ and its derivative is given by
	
	\begin{center}
		$\displaystyle \Psi^\prime(u) v = \sum_{i = 1}^n D_i \Psi(u) v_i$,
	\end{center}
	
	where
	
	\begin{center}
		$\displaystyle D_i \Psi(u) v_i = \int_{\mathbb{R}^N} \frac{\partial F}{\partial u_i}(x, u_1, \dots, u_n) v_i dx$,
	\end{center}
	
for all $\displaystyle v=(v_1,v_2,...,v_n)\in X$. Define also  $\Phi$ in $X$  by
	
	\begin{center}
		$\displaystyle \Phi(u)=\sum_{i=1}^{n}\Phi_i(u_i)$,
	\end{center}
where
$$\Phi_i(u_i)=
\widehat{M_i}\left(\mathcal{B}_i(u_{i}(x))\right),$$
for any $u=(u_1,\dots,u_n)$ in $X$, with $\displaystyle \widehat{M}_i(t):=\int_{0}^{t}M_i(s)ds$ for all $t\geq0$, ($i\in \{1,\dots,n\}$).
	
	We recall that $\Phi$ is a $C^1$-functional, weakly lower semi-continuous and its derivative is given by
	
	\begin{center}
		$\displaystyle \Phi^\prime (u) v =  \sum_{i = 1}^n D_i \Phi(u) v_i$,
	\end{center}
	
	where
	
	\begin{center}
		$\displaystyle D_i \Phi(u) v_i = M_i\left(\mathcal{B}_i(u_i)\right)\int_{\mathbb{R}^N} \Big(\mathcal{A}_{1i}(\nabla u_i) \nabla v_i + w_i(x)\mathcal{A}_{2i}(u_i) v_i \Big)\,dx$,
	\end{center}
	
	for all $v = (v_1, \dots, v_n) \in X$. \\
	
	The Euler-Lagrange functional associated to the system \eqref{s1.1} is defined by
	
	\begin{center}
		$\displaystyle E_{\lambda}(u):= \Phi(u)-\lambda \Psi(u) \quad \forall u \in X$.
	\end{center}
	
	Since $E_{\lambda} \in C^{1}(X, \mathbb{R})$ and for all $v = (v_1, \dots, v_n) \in X$, we have
	
	\begin{center}
		$\displaystyle E_{\lambda}^\prime(u) v = \Phi^\prime (u) v - \lambda \Psi^\prime (u) v$
	\end{center}
	
	Consequently, $u \in X$ is a weak solution of \eqref{s1.1} if and only if $u$ is a critical point of $E_\lambda$. \\
	
	We end this section by stating some hypotheses that we will be using later. \\
	
		\subsection*{\textbf{Hypotheses}} We assume some growth conditions:
	\begin{itemize}
		\item[$(\mathcal{F}_1)$] $F\in C^{1}(\mathbb{R}^N\times \mathbb{R}^n,\mathbb{R})$ and $F(x,0,\dots, 0)=0$.
		
		\item[$(\mathcal{F}_2)$] There exist positive functions $b_{ij}$ ($1\leq i,j \leq n$), such that
		\begin{gather*}
			\Big| \frac{\partial F}{\partial u_i}(x,u_1,...,u_n)\Big| \leq
			\sum_{j=1}^{n} b_{ij}(x)| u_j|^{\mu_{ij}-1},
		\end{gather*}
		where
		$\displaystyle 1<\mu_{ij}<\inf_{x\in \mathbb{R}^N}\gamma_i(x)$ for all $x\in \mathbb{R}^N$ and for all $i\in \left\lbrace 1,2,...,n\right\rbrace$. The weight-functions $b_{ii}$ (resp  $b_{ij}$ if $i\neq j$) belong to the generalized Lebesgue spaces $L^{\alpha _i}(\mathbb{R}^N)$ (resp $L^{\alpha_{ij} }(\mathbb{R}^N)$), with
		\[
		\alpha_i(x)=\frac{\gamma_i(x)}{\gamma_i(x)-1}, \quad \alpha_{ij} (x)
		=\frac{\gamma_i^{\partial }(x)\gamma_j^{\partial }(x)}{\gamma_i^{\partial }(x)\gamma_j^{\partial }(x)-\gamma_i^{\partial
			}(x)-\gamma_j^{\partial }(x)}.
		\]

	\textbf{Example:} 
	We give an  example of potential  $F$ satisfying hypotheses $(\mathcal{F}_1)$ and $(\mathcal{F}_2)$ for $n=2$. Let
	\[
	F(x,u_1,u_2)=a(x)|u_1|^{l_1(x)}|u_2|^{l_2(x)},
	\]
	where $\displaystyle \frac{l_1(x)}{\gamma_1(x)}+\frac{l_2(x)}{\gamma_2(x)}< 1$ and $a$ is a positive function in $L^{e(x)}(\mathbb{R}^N ) $ such that
	
	\begin{center}
	    $\displaystyle e(x)= \dfrac{\gamma_1^{\partial}(x)\gamma_2^{\partial}(x)}{\gamma_1^{\partial}(x)\gamma_2^{\partial}(x) -l_1(x) \gamma_2^{\partial}(x)-l_2(x) \gamma_2^{\partial}(x)}$,
	\end{center}
	
	for each $x\in \mathbb{R}^N $. \\
	
	We can easily verify that $F(x,u_1,u_2)$ satisfies the condition $(\mathcal{F}_1)$. Moreover, by using Young inequality we easily check that the condition $(\mathcal{F}_2)$. \\
	
		\item[$(\mathcal{F}_3)$] Assume that there exist $r>0$ and $z=(z_1,...,z_n)\in X$ such that the following conditions are satisfied:
		\begin{itemize}
			\item[$(C_1)$] $\displaystyle \sum_{i=1}^{n}\frac{m_0^{\star}}{p_i^{+}}\Big(\min \left\lbrace \| z_i\|_{p_i(x),w_i}^{p_i^-} ,\| z_i\|_{p_i(x),w_i}^{p_i^+} \right\rbrace + \mathcal{H}(k_i^3) \min \left\lbrace \| z_i\|_{q_i(x),w_i}^{q_i^-} ,\| z_i\|_{q_i(x),w_i}^{q_i^+} \right\rbrace\Big)>r.$
			
				\begingroup\makeatletter\def \f@size{9}\check@mathfonts
			\item[$(C_2)$] $\displaystyle \frac{\displaystyle \int_{\mathbb{R}^N} \sup_{(\xi_1,\dots,\xi_n) \in K(\frac{sr}{m_0^{\star}})}F(x,\xi_1,\dots,\xi_n)dx }{r} < \dfrac{\displaystyle \int_{\mathbb{R}^N} F(x,z_1, \dots ,z_n) dx}{\displaystyle m_1 \sum_{i=1}^n \Big(\max \left\lbrace \| z_i\|_{p_i(x),w_i}^{p_i^-} ,\| z_i\|_{p_i(x),w_i}^{p_i^+} \right\rbrace + \mathcal{H}(k_i^3) \max \left\lbrace \| z_i\|_{q_i(x),w_i}^{q_i^-} ,\| z_i\|_{q_i(x),w_i}^{q_i^+} \right\rbrace\Big)}$
				\endgroup
		\end{itemize}
		
		where 
		
		\begin{equation}\label{defK}
			\displaystyle K(t):= \left\lbrace (\xi_1,...\xi_n)\in \mathbb{R}^N: \displaystyle \sum_{i=1}^{n} \bigg(\min \left\lbrace |\xi_i|_{p_i^{\partial}(x)}^{{(p_i^{\partial})^{-}}}, |\xi_i|_{p_i^{\partial}(x)}^{{(p_i^{\partial})^{+}}} \right\rbrace + \mathcal{H}(k_i^3) \min \left\lbrace |\xi_i|_{q_i^{\partial}(x)}^{{(q_i^{\partial})^{-}}}, |\xi_i|_{q_i^{\partial}(x)}^{{(q_i^{\partial})^{+}}} \right\rbrace \bigg)\leq t \right\rbrace,
		\end{equation}
		and 
		\begingroup\makeatletter\def \f@size{9}\check@mathfonts
		\begin{equation}\label{defs}
				m_0^{\star} = \displaystyle m_0 \min_{1 \leq i \leq n}\Big\{\displaystyle \frac{\min\big\{\displaystyle \min\{k_{1i}^0,k_{2i}^0\}, \displaystyle \min \{k_{1i}^2,k_{2i}^2\}\big\}}{\displaystyle \max\{\beta_{1i}, \beta_{2i}\}}\Big\}	
				\text{ and }
				\displaystyle s = \max \left\lbrace p_i^+ \min_{1\leq i\leq n} \left\lbrace c_{p_i(x)}^{{p_i^{-}}}, c_{p_i(x)}^{{p_i^{+}}}\right\rbrace , p_i^+ \min_{1\leq i\leq n} \left\lbrace c_{q_i(x)}^{{q_i^{-}}}, c_{q_i(x)}^{{q_i^{+}}}\right\rbrace \right\rbrace ,
		\end{equation}
	\endgroup
		with $t>0$, such that $c_{p_i(x)}$ and $c_{q_i(x)}$ representing the constants defined in Proposition \ref{prop4}. 
	\end{itemize}
	
	\section{Main results}
	
	The following lemmas are needed in the proof of our main results, namely Theorem \ref{Th3.3} below.
	
	\begin{lemma}\label{lem.3.1}
		The functional $\Phi$ is continuously GÃ¢teaux differentiable and sequentially weakly lower semi-continuous, coercive and its GÃ¢teaux derivative admits a continuous inverse on $X^{\star}$.
	\end{lemma}
	
	\begin{proof}
		It is well known that the functional $\Phi$ is well defined and is a continuously GÃ¢teaux differentiable functional whose derivative at the point $u=(u_1,\dots,u_n) \in X$ is the functional $\Phi^\prime(u)$ given by		
		\begin{center}
			$\displaystyle \langle \Phi^{\prime}\left(u\right),\varphi \rangle =\sum_{i=1}^{n}\langle \Phi_{i}^{\prime}\left(u_{i}\right), \varphi_{i} \rangle,$
		\end{center}
		
  where
				\begin{center}
			$\displaystyle \langle \Phi_{i}^\prime(u_i),\varphi_i \rangle = M_i\Big(\mathcal{B}_i(u_i)\Big) \int_{\mathbb R^N} \Big(\mathcal{A}_{1i}(\nabla u_i)\nabla \varphi_i + w_i(x)\mathcal{A}_{2i}(u_i)\varphi_i \Big) dx$,
		\end{center}
		for every $\varphi=(\varphi_1,...,\varphi_n) \in X$ and for all $i \in \{1,\dots,n\}$. Let us show that $\Phi$ is coercive. By using
		$(\textbf{H}_2)$, $(\textbf{H}_4)$ and $(\textbf{M})$, we have for all $u=(u_1,\dots,u_n) \in X$,
		
		\begin{center}
			$\displaystyle \begin{aligned} \Phi(u) & = \sum_{i=1}^n \Phi_{i}(u_i) = \sum_{i=1}^n  \widehat{M}_i\Big(\mathcal{B}_i(u_i)\Big) \\
				 & \geq \sum_{i=1}^n \frac{m_0}{p_i^+} \int_{\mathbb R^N} \Big(\mathcal{A}_{1i}(|\nabla u_i|) + w_i(x)\mathcal{A}_{2i}(|u_i|) \Big) dx \\
				 & \geq \sum_{i=1}^n \frac{m_0}{p_i^+} \int_{\mathbb R^N} \Big[\frac{1}{\beta_{1i}} a_{1i}(|\nabla u_i|^{p_i(x)})|\nabla u_i|^{p_i(x)} + \frac{1}{\beta_{2i}} w_i(x)a_{2i}(|u_i|^{p_i(x)})|u_i|^{p_i(x)} \Big] dx
				  \\ & \geq \sum_{i=1}^n  \frac{m_0}{p_i^+\max\{\beta_{1i}, \beta_{2i}\}} \int_{\mathbb R^N} \Big[a_{1i}(|\nabla u_i|^{p_i(x)})|\nabla u_i|^{p_i(x)} + w_i(x)a_{2i}(|u_i|^{p_i(x)})|u_i|^{p_i(x)} \Big] dx \\
				   & \geq \sum_{i=1}^n  \frac{m_0\min\{\min\{k_{1i}^0,k_{2i}^0\}, \min\{k_{1i}^2,k_{2i}^2\}\}}{p_i^+\max\{\beta_{1i}, \beta_{2i}\}} \cdot \Big(\rho_{p_i(x)w_i}(u_i) + \mathcal{H}(k_{i}^3) \cdot\rho_{q_i(x)w_i}(u_i) \Big),
				 \end{aligned}$
			\end{center}
		
		Hence, by inequalities \eqref{inq1} and $\eqref{inq2}$, we obtain
	\begin{center}
		$\displaystyle \begin{aligned} \Phi(u) & \geq \sum_{i=1}^n \frac{m_0\min\{\min\{k_{1i}^0,k_{2i}^0\}, \min\{k_{1i}^2,k_{2i}^2\}\}}{p_i^+\max\{\beta_{1i}, \beta_{2i}\}} \cdot \Big(\min\{\| u_i\|_{p_i(x),w_i}^{p_i^-} , \| u_i\|_{p_i(x),w_i}^{p_i^+}\}+ \mathcal{H}(k_{i}^3) \cdot \min\{\| u_i\|_{q_i(x),w_i}^{q_i^-}, \| u_i\|_{q_i(x),w_i}^{q_i^+}\} \Big). \end{aligned}$
		\end{center}
		
		This shows that $\Phi(u) \longrightarrow +\infty$ as $\| u\| \longrightarrow +\infty$ that is, $\Phi$ is coercive on $X$. Now, in order to show that the functional $\Phi^\prime: X \longrightarrow X^{\star}$ is strictly monotone, it suffices to prove that $\Phi$ is strictly convex. \newline

For all $i \in \{1,\dots,n\}$, the functional $\mathcal{B}_i : X_i \longrightarrow \mathbb{R}$ defined in \eqref{AB} is clearly a GÃ¢teaux differentiable at any $u_i \in X_i$, and its derivative is given by
		
		\begin{align*}
		 \displaystyle\langle \mathcal{B}_i^\prime(u_i),\varphi_i \rangle
			&= \int_{\mathbb R^N} \Big(\mathcal{A}_{1i}(|\nabla u_i|^{p_i(x)})\nabla \varphi_i + w_i(x)\mathcal{A}_{2i}(|u_i|^{p_i(x)})\varphi_i \Big) dx\\
			&= \int_{\mathbb R^N}\Big( a_{1i}(| \nabla u_i|^{p_i(x)}) | \nabla u_i|^{p_i(x)-2} \nabla u_i .\nabla\varphi_i 
			+ w_i(x)a_{2i}(| u_i|^{p_i(x)}) | u_i|^{p_i(x)-2} u_i\varphi_i\Big)dx,
		\end{align*}
for all $\varphi_i\in X_i$. \\

	Taking into account the elementary inequalities (see, e.g., Auxiliary Results in \cite{Hurtado}) for any $ \varrho , \zeta \in \mathbb{R}^N $ 
	\begin{equation}\label{ineq1}
		\Big( a_{ji}(|\varrho|^{p_i(x)})|\varrho|^{p_i(x)-2}\varrho  - a_{ji}(|\zeta|^{p_i(x)})|\zeta|^{p_i(x)-2}\zeta 
		\Big) \cdot\left(\varrho -\zeta\right)
		\geq \begin{cases}
			C_{p_i}|\varrho  -\zeta|^{p_i(x)} &\text{if } p_i(x)\geq2\\[4pt]
			C_{p_i}\frac{	|\varrho -\zeta|^{2}}{(|\varrho| +|\zeta| )^{p_i(x)-2}},\;(\varrho,\zeta)\neq(0,0) &\text{if }1<p_i(x)<2,
		\end{cases}
	\end{equation}
	
	where $\cdot$ denotes the standard inner product  in $\mathbb{R}^N$. Therefore, we have
	
		\begin{center}
			$\displaystyle \langle 	\mathcal{B}_i^\prime(u_i) - 	\mathcal{B}_i^\prime(v_i),u_i - v_i \rangle > 0$,
		\end{center}	
for all  $u_i\neq v_i\in X_i$,
which means that $ \mathcal{B}_i'$ is strictly monotone. So, by  \cite[ Proposition 25.10]{Zeidler}, 
$\mathcal{B}_i$ is strictly convex. Moreover, since  the Kirchhoff function $M_i$ is nondecreasing, $\widehat{M}_i$ is convex in
$[0,+\infty[$. Thus, for every $u_i, v_i\in X_i$ with $u_i\neq v_i$, and every
$s,t\in (0,1)$ with $s+t=1$, we have
$$
\widehat{M}_i(\mathcal{B}_i(su_i+tv_i))<\widehat{M}_i(s\mathcal{B}_i(u_i)+t\mathcal{B}_i(v_i))\leq
s\widehat{M}_i(\mathcal{B}_i(u_i))+t\widehat{M}_i(\mathcal{B}_i(v_i)).
$$
	This shows that $\Phi_{i}$ is strictly convex in $ W_{w_i}^{1, p_i(x)}\left(\mathbb{R}^N\right) \cap W^{1, \gamma_i(x)}_{w_i}\left(\mathbb{R}^N\right)$ for all $i \in \{1,\dots,n\}$. Hence, $\Phi$ is strictly convex in $X$, and therefore $\Phi^\prime= \displaystyle\sum_{i=1}^{n} \Phi_{i}^\prime$ is strictly monotone.\\
	
	It's clear that $\Phi^\prime$ is an injection since $\Phi^\prime$ is strictly monotone operator in $X$.\newline
	 Moreover, since we have
		
		\begin{center}
			$\displaystyle \begin{aligned} \displaystyle \lim_{\|u\| \longrightarrow + \infty} \frac{\langle \Phi^\prime(u),u \rangle}{\|u\|} & = \displaystyle \lim_{\|(u_1,...,u_n)\| \to + \infty} \frac{\displaystyle\sum_{i=1}^{n} M_i\left(\mathcal{B}_i(u_i)\right)\int_{\mathbb{R}^N} \Big(\mathcal{A}_{1i}(\nabla u_i) \nabla u_i + w_i(x)\mathcal{A}_{2i}(u_i) u_i \Big)\,dx}{\|(u_1,...,u_n)\|}\\
				 & \geq \displaystyle \lim_{\|(u_1,...,u_n)\| \to + \infty} \frac{m_0 \displaystyle\sum_{i=1}^{n} \int_{\mathbb{R}^N} \Big(a_{1i}(| \nabla u_i|^{p_i(x)}) | \nabla u_i|^{p_i(x)} + w_i(x)a_{2i}(| u_i|^{p_i(x)}) |  u_i|^{p_i(x)} \Big)\,dx}{\|(u_1,...,u_n)\|}\\
				  & \geq \displaystyle \lim_{\|(u_1,...,u_n)\| \to+ \infty} \frac{m_0\displaystyle\min_{1\leq i\leq n}\Big\{\min\{k_{1i}^0,k_{2i}^0\}, \min\{k_{1i}^2,k_{2i}^2\}\Big\}\sum_{i=1}^{n} \Big( \| u_i\|_{p_i(x),w_i}^{p_i^+} + \mathcal{H}(k_{i}^3) \cdot \| u_i\|_{q_i(x),w_i}^{q_i^+} \Big) }{ \displaystyle\max_{1\leq i\leq n} \left\lbrace 	\|  u_i\|_{p_i(x),w_i}+ \mathcal{H}(\kappa_i^3)\| u_i\|_{q_i(x),w_i}\right\rbrace} =+\infty.
			 \end{aligned}$
		\end{center}
		
		Then, we deduce that $\Phi^\prime$ is coercive. Thus $\Phi^\prime$ is a surjection. Now, since $\Phi^\prime$ is semicontinuous in $X$, then by applying Minty-Browder theorem (Theorem 26.A of \cite{Zeidler}), we conclude that $\Phi^\prime$ admits a continuous inverse on $X^{\star}$. Moreover, the monotonicity of $\Phi^\prime$ on $X^{\star}$ ensures that $\Phi$ is sequentially lower semi-continuously on $X$ (see \cite{Zeidler}, Proposition 25. 20). The proof of the lemma is complete.
	\end{proof}
	
	\begin{lemma}\label{lem.3.2}
		Under assumptions $(\mathcal{F}_1)$ and $(\mathcal{F}_2)$, the functional $\Psi$ is well defined and is of class $C^1$ on $X$. Moreover, its derivative is given by
		
		\begin{center}
			$\displaystyle \displaystyle\langle\Psi^\prime(u), \varphi\displaystyle\rangle = \sum_{i=1}^n \frac{\partial F}{\partial u_i}(x,u) \varphi_i$, $\quad$ $\forall u=(u_1,\dots,u_n), \varphi=(\varphi_1,\dots,\varphi_n) \in X$.
		\end{center}
		
		Furthermore, $\Psi^\prime$ is compact from $X$ to $X^{\star}$.
	\end{lemma}
	
	\begin{proof}
	We start by showing that the functional $F$ is well defined on $X$. Indeed, for all $u=(u_1,\dots,u_n) \in X$, we have in virtue of $(\mathcal{F}_1)$ and $(\mathcal{F}_2)$,
		
		\begin{center}
			$\displaystyle \begin{aligned} F(x,u) & = \sum_{i=1}^n \int_{0}^{u_i} \frac{\partial F}{\partial s}(x, u_1, \dots, s, \dots, u_n) ds + F(x, 0, \dots, 0) dx \\ & \leq c_{1}\left[\sum_{i=1}^{n}\left(\sum_{j=1}^{n} b_{i j}(x)\left|u_{j}(x)\right|^{\mu_{i j}-1}\left|u_{i}(x)\right|\right)\right], \end{aligned}$
		\end{center}
		
		Then,
		
		\begin{center}
			$\displaystyle \int_{\mathbb{R}^{N}} F\left(x, u_{1}, \ldots, u_{n}\right) d x \leq c_{2}\left[\sum_{i=1}^{n}\left(\int_{\mathbb{R}^{N}} \sum_{j=1}^{n} b_{i j}(x)\left|u_{j}(x)\right|^{\mu_{i j}-1}\left|u_{i}(x)\right| d x\right)\right].$
		\end{center}
		
		If we consider the fact that $W^{1, \gamma(x)}\left(\mathbb{R}^{N}\right) \hookrightarrow L^{\mu(x)}\left(\mathbb{R}^{N}\right)$, for $\mu(x)>1$, then there exists $c>0$ such that
		
		\begin{center}
			$\displaystyle ||u|^{\mu}|_{\gamma(x)} = |u|_{\mu \gamma(x)}^{\mu} \leq c \| u \|_{\gamma(x)}^{\mu}$,
		\end{center}
		
		and if we apply Propositions \ref{prop1},\ref{prop2} and \ref{prop5} and take $b_{i i} \in L^{\alpha_{i}(x)}, b_{i j} \in L^{\alpha_{i j}(x)}$ if $i \neq j$, then we have
		
		\begin{align}
			\displaystyle  \int_{\mathbb{R}^{N}} F(x, u_{1}, \ldots, u_{n}) & \leq c_{3} \Big[\sum_{i=1}^{n} \Big(\sum_{j=1}^{n}|b_{i j}|_{\alpha_{i j}(x)}|| u_{j}|^{\mu_{i j}-1}|_{\gamma_{j}^{\partial}(x)}|u_{i}|_{\gamma_{i}^{\partial}(x)}\Big)\Big] \quad \quad (3.1) \\ & \leq c_{3} \Big[\sum_{i=1}^{n} \Big(\sum_{j=1}^{n}|b_{i j}|_{\alpha_{i j}(x)} |u_{j}|_{(\mu_{i j}-1) \gamma_{j}^{\partial}(x)}^{\mu_{i j}-1} |u_{i}|_{\gamma_{i}^{\partial}(x)}\Big)\Big] \quad  \quad (3.2) \\ 
				& \leq c_{3} \Big[\sum_{i=1}^{n} \Big(\sum_{j=1}^{n}|b_{i j}|_{\alpha_{i j}(x)}\|u_{j}\|_{\gamma_{j}(x)}^{\mu_{i j}-1}\|u_{i}\|_{\gamma_{i}(x)}\Big)\Big] < \infty.  \quad \label{(3.3)} \quad(3.3)
			\end{align}.
		\
		
		Hence, $\Psi$ is well defined. Moreover, one can easily see that $\Psi^{\prime}$ is also well defined on $X$. Indeed, using $(\mathcal{F}_2)$ for all $\varphi=\left(\varphi_{1}, \dots, \varphi_{n}\right) \in X$, we have
		
		\begin{center}
			$\displaystyle \begin{aligned} \Psi^{\prime}(u) \varphi &=\sum_{i=1}^{n} \int_{\mathbb{R}^{N}} \frac{\partial F}{\partial u_{i}}\left(x, u_{1}, \ldots, u_{n}\right) \varphi_{i} d x \\ & \leq \sum_{i=1}^{n}\left(\sum_{j=1}^{n} b_{i j}(x)\left|u_{j}(x)\right|^{\mu_{i j}-1}\right)\left|\varphi_{i}(x)\right| dx. \end{aligned}$
		\end{center}
		
		Following HÃ¶lder inequality, we obtain
		
		\begin{center}
			$\displaystyle \Psi^{\prime}(u) \varphi \leq c_{4} \Big[\sum_{i=1}^{n} \Big(\sum_{j=1}^{n}|b_{i j}|_{\alpha_{i j}(x)}|| u_{j}|^{\mu_{i j}-1}|_{\gamma_{j}^{\partial}(x)}|\varphi_{i}|_{\gamma_{i}^{\partial}(x)}\Big)\Big]$.
		\end{center}
		
		The above propositions yield
		
		\begin{center}
			$\displaystyle \Psi^{\prime}(u) \varphi \leq c\left[\sum_{i=1}^{n}\left(\sum_{j=1}^{n}\left|b_{i j}\right|_{\beta_{i j}(x)}\left\|u_{j}\right\|_{\gamma_{j}(x)}^{\mu_{i j}-1}\left\|\varphi_{i}\right\|_{\gamma_{i}(x)}\right)\right]<\infty$.
		\end{center}
		
		Now let us show that $\Psi$ is differentiable in the sense of Frechet, that is, for fixed $u=\left(u_{1}, \ldots, u_{n}\right) \in X$ and given $\varepsilon>0$, there must be a $\delta=\delta_{\varepsilon, u_{1}, \ldots, u_{n}}>0$ such that
		
		\begin{center}
			$\displaystyle \left|\Psi\left(u_{1}+\varphi_{1}, \ldots, u_{n}+\varphi_{n}\right)-\Psi\left(u_{1}, \ldots, u_{n}\right)-\Psi^{\prime}\left(u_{1}, \ldots, u_{n}\right)\left(\varphi_{1}, \ldots, \varphi_{n}\right)\right| \leq \varepsilon \sum_{i=1}^{n}\left(\left\|\varphi_{i}\right\|_{\gamma_{i}(x)}\right),$
		\end{center}
		
		for all $\varphi=\left(\varphi_{1}, \ldots, \varphi_{n}\right) \in X$ with $\displaystyle \sum_{i=1}^{n}\left(\left\|\varphi_{i}\right\|_{\gamma_{i}(x)}\right) \leq \delta$.
		
		Let $B_{R} = \Big\{ x \in \mathbb{R}^N: |x| < R \Big\}$ be the ball of radius $R$ which is centered at the origin of $\mathbb{R}^{N}$ and denote $B_{R}^{\prime}=\mathbb{R}^{N} \backslash B_{R}$. Moreover, let us define the functional $\Psi_{R}: \displaystyle \prod_{i=1}^{n} \Big(W_{w_{i}}^{1, p_{i}(x)}\left(B_{R}\right)\cap W_{w_{i}}^{1, \gamma_{i}(x)}\left(B_{R}\right)\Big) \longrightarrow \mathbb{R}$ as follows:
		
		\begin{center}
			$\displaystyle \Psi_{R}(u)=\int_{B_{R}} F\left(x, u_{1}(x), \ldots, u_{n}(x)\right) dx, \quad \forall u=(u_1,\dots,u_n)$.
		\end{center}
		
		If we consider $(\mathcal{F}_1)$ and $(\mathcal{F}_2)$, it is easy to see that $\displaystyle \Psi_{R} \in C^{1}\Big(\prod_{i=1}^{n}  W_{w_{i}}^{1, p_{i}(x)}\left(B_{R}\right)\cap W_{w_{i}}^{1, \gamma_{i}(x)}\left(B_{R}\right)\Big)$, and in addition for all $\displaystyle \varphi=\left(\varphi_{1}, \ldots, \varphi_{n}\right) \in \prod_{i=1}^{n}  \Big(W_{w_{i}}^{1, p_{i}(x)}\left(B_{R}\right)\cap W_{w_{i}}^{1, \gamma_{i}(x)}\left(B_{R}\right)\Big)$, we have
		
		\begin{center}
			$\displaystyle \Psi_{R}^{\prime}(u) \varphi=\sum_{i=1}^{n} \int_{B_{R}} \frac{\partial F}{\partial u_{i}}\left(x, u_{1}(x), \ldots, u_{n}(x)\right) \varphi_{i}(x) dx$.
		\end{center}
		
		Also as we know, the operator $\displaystyle \Psi_{R}^{\prime}: \prod_{i=1}^{n}  \Big(W_{w_{i}}^{1, p_{i}(x)}\left(B_{R}\right)\cap W_{w_{i}}^{1, \gamma_{i}(x)}\left(B_{R}\right)\Big) \longrightarrow \Big(\prod_{i=1}^{n}  \Big(W_{w_{i}}^{1, p_{i}(x)}\left(B_{R}\right)\cap W_{w_{i}}^{1, \gamma_{i}(x)}\left(B_{R}\right)\Big)\Big)^{\star}$ is compact \cite{Fan3}. Then, for all $u=\left(u_{1}, \ldots, u_{n}\right), \varphi=\left(\varphi_{1}, \ldots, \varphi_{n}\right) \in X$, we can write
		
		\begin{center}
			$\displaystyle \begin{aligned} & \left|\Psi\left(u_{1}+\varphi_{1}, \ldots, u_{n}+\varphi_{n}\right)-\Psi\left(u_{1}, \ldots, u_{n}\right)-\Psi^{\prime}\left(u_{1}, \ldots, u_{n}\right)\left(\varphi_{1}, \ldots, \varphi_{n}\right)\right| \\
		 & \leq\left|\Psi_{R}\left(u_{1}+\varphi_{1}, \ldots, u_{n}+\varphi_{n}\right)-\Psi_{R}\left(u_{1}, \ldots, u_{n}\right)-\Psi_{R}^{\prime}\left(u_{1}, \ldots, u_{n}\right)\left(\varphi_{1}, \ldots, \varphi_{n}\right)\right| \\ 
		&\quad + \Big| \int_{B_{R}^{\prime}}\left(F\left(x, u_{1}+\varphi_{1}, \ldots, u_{n}+\varphi_{n}\right)-F\left(x, u_{1}, \ldots, u_{n}\right)-\right. \left.\sum_{i=1}^{n} \int_{B_{R}} \frac{\partial F}{\partial u_{i}}\left(x, u_{1}, \ldots, u_{n}\right) \varphi_{i}\right) d x \Big| .
	\end{aligned}$
		\end{center}
		
		By virtue of mean-value theorem, there exist $\left.\xi_{1}, \ldots, \xi_{n} \in \right ]0,1[$ such that
		
		\begin{center}
			$\displaystyle \begin{gathered} \left| \int_{B_{R}^{\prime}}\left(F\left(x, u_{1}+\varphi_{1}, \ldots, u_{n}+\varphi_{n}\right)-F\left(x, u_{1}, \ldots, u_{n}\right)\right) - \sum_{i=1}^{n} \int_{B_{R}} \frac{\partial F}{\partial u_{i}}\left(x, u_{1}, \ldots, u_{n}\right) \varphi_{i} dx \right| \\ = \left|\int_{B_{R}^{\prime}}\left(\sum_{i=1}^{n} \frac{\partial F}{\partial u_{i}}\left(x, u_{1}, \ldots, u_{i}+\xi_{i} \varphi_{i}, \ldots, u_{n}\right) \varphi_{i}-\sum_{i=1}^{n} \frac{\partial F}{\partial u_{i}}\left(x, u_{1}, \ldots, u_{n}\right)\right) dx \right| .\end{gathered}$
		\end{center}
		
		Using the condition $(\mathcal{F}_2)$, we have
		
		\begin{center}
			$\displaystyle \left|\int_{B_{R}^{\prime}}\left(F\left(x, u_{1}+\varphi_{1}, \ldots, u_{n}+\varphi_{n}\right)-F\left(x, u_{1}, \ldots, u_{n}\right)\right)-\sum_{i=1}^{n} \int_{B_{R}} \frac{\partial F}{\partial u_{i}}\left(x, u_{1}, \ldots, u_{n}\right) \varphi_{i} d x\right| \leq\left|\sum_{i=1}^{n}\left(\sum_{j=1}^{n} \int_{B_{R}^{\prime}} b_{i j}(x)\left(\left|u_{j}+\xi_{j} \varphi_{j}\right|^{\mu_{i j}-1}-\left|u_{j}\right|^{\mu_{i j}-1}\right) \varphi_{i} d x\right)\right|.$
		\end{center}
		
		Using the elementary inequality $|a+b|^{s} \leq 2^{s-1}\left(|a|^{s}+|b|^{s}\right)$ for $a, b \in \mathbb{R}^{N}$, we can write
		
		\begin{center}
			$\displaystyle \begin{gathered} \leq \sum_{i=1}^{n}\left(\sum _ { j = 1 } ^ { n } \left(\left(2^{\mu_{i j}-1}-1\right) \int_{B_{R}^{\prime}} b_{i j}(x)\left|u_{j}\right|^{\mu_{i j}-1}\left|\varphi_{i}\right| dx \right.\left.+\left(\xi_{j} 2\right)^{\mu_{i j}-1} \int_{B_{R}^{\prime}} b_{i j}(x)\left|\varphi_{j}\right|^{\mu_{i j}-1}\left|\varphi_{i}\right| d x\right)\right) \end{gathered}$.
		\end{center}
		
		Then, applying Propositions \ref{prop2}, \ref{prop3} and \ref{prop5}, we have
		
		\begin{center}
			$\displaystyle \leq \sum_{i=1}^{n} c\left(\sum_{j=1}^{n}\left(\left|b_{i j}(x)\right|_{\alpha_{i j}}\left\|u_{j}\right\|_{\gamma_{1}^{\partial}(x)}^{\mu_{i j}-1}+\left|b_{i j}(x)\right|_{\alpha_{i j}}\left\|\varphi_{j}\right\|_{\gamma_{j}^{\partial}(x)}^{\mu_{i j}-1}\right)\right)\left\|\varphi_{i}\right\|_{\gamma_{i}(x)}$
		\end{center}
		
		and by the fact that
		
		\begin{center}
			$\displaystyle \left|b_{i i}(x)\right|_{L^{\alpha_{i}}\left(B_{R}^{\prime}\right)} \longrightarrow 0$,
		\end{center}
		
		and
		
		\begin{center}
			$\displaystyle \left|b_{i j}(x)\right|_{L^{\alpha_{i}}\left(B_{R}^{\prime}\right)} \longrightarrow 0$,
		\end{center}
		
		for all $1 \leq i, j \leq n$, as $R \rightarrow \infty$, and for $R$ sufficiently large, we obtain the estimate
		
		\begin{center}
			$\displaystyle \Big| \int_{B_{R}^{\prime}}\left(F\left(x, u_{1} + \varphi_{1}, \ldots, u_{n} + \varphi_{n}\right) - F\left(x, u_{1}, \ldots, u_{n}\right) - \sum_{i=1}^{n} \frac{\partial F}{\partial u_{i}}\left(x, u_{1}, \ldots, u_{n}\right) \varphi_{i}\right) dx \Big| \leq \varepsilon \sum_{i=1}^{n}\left(\left\|\varphi_{i}\right\|_{\gamma_{i}(x)}\right) $.
		\end{center}
		
	It only remains to show that $\Psi^{\prime}$ is continuous on $X$. Let $u_{m}=\left(u_{m,1}, \ldots, u_{m,n}\right)$ be such that $u_{m} \rightarrow u$ as $m \rightarrow \infty$. Then, for $\varphi = \left(\varphi_{1}, \ldots, \varphi_{n}\right) \in X$, we have
		
		\begin{center}
			$\displaystyle \begin{aligned} \left|\Psi^{\prime}\left(u_{m}\right) \varphi - \Psi^{\prime}(u) \varphi \right| & \leq \left|\Psi_{R}^{\prime}\left(u_{m}\right) \varphi -\Psi_{R}^{\prime}(u) \varphi \right| \\ & + \sum_{i=1}^{n} \int_{B_{R}^{\prime}}\left|\left(\frac{\partial F}{\partial u_{i}}\left(x, u_{m,1}, \ldots, u_{m,n}\right) \varphi_{i}-\frac{\partial F}{\partial u_{i}}\left(x, u_{1}, \ldots, u_{n}\right) \varphi_{i}\right) dx \right| \end{aligned}$
		\end{center}
		
		Since $\Psi_{R}^{\prime}$ is continuous on $\displaystyle \prod_{i=1}^{n} \Big(W_{w_{i}}^{1, p_{i}(x)}\left(B_{R}\right)\cap W_{w_{i}}^{1, \gamma_{i}(x)}\left(B_{R}\right)\Big)$ (see \cite{Fan3}), we have
		
		\begin{center}
			$\displaystyle \left|\Psi_{R}^{\prime}\left(u_{m}\right) \varphi -\Psi_{R}^{\prime}(u) \varphi\right| \longrightarrow 0\quad \text{ as } m \rightarrow \infty$.
		\end{center}
		
		Now, using $(\mathcal{F}_2)$ once again and taking into account that the other terms on the right-hand side of the above inequality tend to zero, we conclude that $\Psi^{\prime}$ is continuous on $X$. \\
		
		As for the compactness of $\Psi^{\prime}$, let $u_{m}=\left(u_{m,1}, \ldots, u_{m,n}\right)$ be a bounded sequence in $X$. Then, there exists a subsequence (we also denote it as $u_{m}=\left(u_{m,1}, \ldots, u_{m,n}\right)$) which converges weakly in $X$ to $u=\left(u_{1}, \ldots, u_{n}\right) \in X$. Then, if we use the same arguments as above, we have
		
		\begin{center}
			$\displaystyle \begin{aligned} \left|\Psi^{\prime}\left(u_{m}\right) \varphi -\Psi^{\prime}(u) \varphi\right| & \leq \left|\Psi_{R}^{\prime}\left(u_{m}\right) \varphi -\Psi_{R}^{\prime}(u) \varphi \right| \\ & + \sum_{i=1}^{n} \int_{B_{R}^{\prime}}\left|\left(\frac{\partial F}{\partial u_{i}}\left(x, u_{m,1}, \ldots, u_{m,n} \right) \varphi_{i}-\frac{\partial F}{\partial u_{i}}\left(x, u_{1}, \ldots, u_{n}\right) \varphi_{i}\right) d x\right|. \end{aligned}$
		\end{center}
		
		Since the restriction operator is continuous, we have $u_{m} \rightharpoonup u$ in  $\displaystyle \prod_{i=1}^{n} \Big(W_{w_{i}}^{1, p_{i}(x)}\left(B_{R}\right)\cap W_{w_{i}}^{1, \gamma_{i}(x)}\left(B_{R}\right)\Big)$. Because of the compactness of $\Psi^{\prime}$, the first expression on the right-hand side of the inequality tends to 0 , as $m \longrightarrow \infty$, and, as above, for sufficiently large $R$ we obtain
		
		\begin{center}
			$\displaystyle \sum_{i=1}^{n} \int_{B_{R}^{\prime}}\left|\left(\frac{\partial F}{\partial u_{i}}\left(x,u_{m,1}, \ldots, u_{m,n}\right) \varphi_{i}-\frac{\partial F}{\partial u_{i}}\left(x, u_{1}, \ldots, u_{n}\right) \varphi_{i}\right) \right| dx \longrightarrow 0.$
		\end{center}
		
		This implies $\Psi^{\prime}$ is compact from $X$ to $X^{\star}$. 
	\end{proof}
	
	\begin{theorem}\label{Th3.3}
		Under assumptions $(\mathcal{F}_1)-(\mathcal{F}_3)$, the system \eqref{s1.1} admits at least three distinct weak solutions in $X$ for each $\lambda$ that belongs to the interval
		
		\begin{center}
			$\displaystyle \Lambda_{r}= \Bigg] \frac{\displaystyle m_1 \sum_{i=1}^n \max \left\lbrace \| z_i\|_{p_i(x),w_i}^{p_i^-} ,\| z_i\|_{p_i(x),w_i}^{p_i^+} \right\rbrace + \mathcal{H}(k_i^3) \max \left\lbrace \| z_i\|_{q_i(x),w_i}^{q_i^-} ,\| z_i\|_{q_i(x),w_i}^{q_i^+} \right\rbrace}{\displaystyle \int_{\mathbb{R}^N} F(x, z_1, \dots, z_n) dx} , \frac{r}{\displaystyle \int_{\mathbb{R}^N} \sup_{(\xi_1, \dots, \xi_n) \in K(\frac{sr}{m_0^{\star}})} F(x, \xi_1, \dots, \xi_n) dx} \Bigg[.$
		\end{center}
	\end{theorem}
	
	\begin{proof}
		By Lemma \ref{lem.3.1}, $\Phi$ is coercive, and by definitions of $\Phi$ and $\Psi$ and from hypothesis $(\mathcal{F}_1)$, we have $\Phi(0,\dots,0) = \Psi(0, \dots, 0) = 0$. Moreover, the required hypothesis $\Phi(\bar x) > r$ follows from $(C_1)$ and the definition of $\Phi$ by choosing $\bar x = (z_1, \dots, z_n)$.	On the contrary, by applying Proposition \ref{prop4} for $u=(u_1, \dots, u_n) \in X$, we have
		
			\begin{multline}\label{inq32}
				\displaystyle \frac{1}{s} \Bigg( \sum_{i=1}^n   \Big( \min\Big\{|u_i|_{p_i^{\partial}}^{p_i^-}, |u_i|_{p_i^{\partial}}^{p_i^+}\Big\} + \mathcal{H}(\kappa_i^3) \cdot \min\Big\{ |u_i|_{q_i^{\partial}}^{q_i^-}, |u_i|_{q_i^{\partial}}^{q_i^+} \Big\} \Big) \Bigg) \\ \leq \sum_{i=1}^n \frac{1}{p_i^+}  \Big(\min\Big\{\| u_i\|_{p_i(x),w_i}^{p_i^-} , \| u_i\|_{p_i(x),w_i}^{p_i^+}\Big\}  +\mathcal{H}(k_{i}^3) \cdot \min\Big\{\| u_i\|_{q_i(x),w_i}^{q_i^-}, \| u_i\|_{q_i(x),w_i}^{q_i^+}\Big\} \Big) .
				\end{multline}
		
		with $\displaystyle s = \max_{1 \leq i \leq n}\Big\{p_i^+ \displaystyle \min \{c_{p_i(x)}^{p_i^{-}}, c_{p_i(x)}^{p_i^{+}} \}, p_i^+ \displaystyle \min \{c_{q_i(x)}^{q_i^{-}}, c_{q_i(x)}^{q_i^{+}} \}\Big \} $, defined in $\eqref{defs}$. \\
		
		Now, from  \eqref{inq32}, we obtain for $r > 0$
		
		\begin{center}
			$\displaystyle \begin{aligned} \Phi^{-1}(]-\infty,r]) & = \Big\{u = (u_1, \dots, u_n) \in X: \Phi(u_1, \dots, u_n) \leq r \Big\} \\ 
				& \subseteq \Big\{u = (u_1, \dots, u_n) \in X: \displaystyle \sum_{i=1}^n  \frac{m_i^0\min\{\displaystyle \min\{k_{1i}^0,k_{2i}^0\}, \displaystyle \min \{k_{1i}^2,k_{2i}^2\}\}}{p_i^+\displaystyle \max\{\beta_{1i}, \beta_{2i}\}} \\ &
				\quad \cdot \Big(\min\{\| u_i\|_{p_i(x),w_i}^{p_i^-} , \| u_i\|_{p_i(x),w_i}^{p_i^+}\} + \mathcal{H}(k_{i}^3) \cdot \min\{\| u_i\|_{q_i(x),w_i}^{q_i^-}, \| u_i\|_{q_i(x),w_i}^{q_i^+}\} \Big) \leq r \Big\} \\ 
				& \subseteq \displaystyle\Big\{\frac{\displaystyle m_0 \min_{1 \leq i \leq n}\{\kappa_i\} }{\displaystyle\max_{1 \leq i \leq n}\Big\{p_i^+ \displaystyle \min \{c_{p_i(x)}^{p_i^{-}}, c_{p_i(x)}^{p_i^{+}} \}, p_i^+ \displaystyle \min \{c_{q_i(x)}^{q_i^{-}}, c_{q_i(x)}^{q_i^{+}} \}\Big \}} \sum_{i=1}^{n} \Big( \min\{ |\xi_i|_{p_i^{\partial}(x)}^{{p_i^{-}}}, |\xi_i|_{p_i^{\partial}(x)}^{{p_i^{+}}} \} \\ 
				& + \mathcal{H}(k_i^3) \min \{ |\xi_i|_{q_i^{\partial}(x)}^{{q_i^{-}}}, |\xi_i|_{q_i^{\partial}(x)}^{{q_i^{+}}} \} \Big) \leq r \Big\} \\
				 & = \displaystyle K\Big(\frac{sr}{m_0^{\star}}\Big) \end{aligned}$
		\end{center}
		
		where $K(t)$ is defined in $\eqref{defK}$, and $\kappa_i$ and $m_0^{\star}$ are defined as follows
		
		\begin{center}
			$\displaystyle \begin{aligned} \kappa_i & = \displaystyle \frac{\min\{\displaystyle \min\{k_{1i}^0,k_{2i}^0\}, \displaystyle \min \{k_{1i}^2,k_{2i}^2\}\}}{\displaystyle \max\{\beta_{1i}, \beta_{2i}\}},\\
				 m_0^{\star} & = \displaystyle m_0 \min_{1 \leq i \leq n}\{\kappa_i\} .\end{aligned}$
		\end{center}
		
Then, we get  
		
		$\displaystyle \begin{aligned} \sup_{(u_1, \dots, u_n) \in \Phi^{-1}(]-\infty, r])} \Psi (u_1, \dots, u_n) & = \sup_{(\xi_1, \dots, \xi_n) \in K(\frac{sr}{m_0^{\star}})} \int_{\mathbb{R}^N} F(x, u_1, \dots, u_n) dx \\ & \leq \int_{\mathbb{R}^N} \sup_{(\xi_1, \dots, \xi_n) \in K(\frac{sr}{m_0^{\star}})} F(x, \xi_1, \dots, \xi_n) dx .\end{aligned}$ \\
		
		Therefore from $(C_2)$, we have
		
		$\displaystyle \begin{aligned} \sup_{(u_1, \dots, u_n) \in \Phi^{-1}(]-\infty, r])} \Psi (u_1, \dots, u_n) & \leq r \frac{\displaystyle \int_{\mathbb{R}^N} F(x, z_1, \dots, z_n) dx}{\displaystyle m_1 \sum_{i=1}^n \Big(\max \left\lbrace \| z_i\|_{p_i(x),w_i}^{p_i^-} ,\| z_i\|_{p_i(x),w_i}^{p_i^+} \right\rbrace + \mathcal{H}(k_i^3) \max \left\lbrace \| z_i\|_{q_i(x),w_i}^{q_i^-} ,\| z_i\|_{q_i(x),w_i}^{q_i^+} \right\rbrace\Big) }
			\\ & \leq r\frac{\Psi(z_1, \dots, z_n)}{\Phi(z_1, \dots, z_n)}, \end{aligned}$ 
		
		 from which condition $(a_1)$ of Lemma \ref{lem1.1} follows. 
		 
		 In the next step, we shall prove that for each $\lambda>0$, the energy functional $E_\lambda = \Phi - \lambda \Psi$ is coercive. For all $u=(u_1, \dots, u_n) \in X$, under the assumptions $(\textbf{H}_2)$ and $(\textbf{H}_3)$ and by using inequality \eqref{(3.3)}, we have 
		
		\begin{center}
			$\displaystyle \begin{aligned} E_\lambda (u) & = \sum_{i=1}^n \widehat{M}_i\Big(\mathcal{B}_i(u_i)\Big) - \lambda \int_{\mathbb{R}^N} F(x, u_1, \dots, u_n) dx \\ 
				& \geq \sum_{i=1}^n \frac{m_0}{p_i^+} \int_{\mathbb R^N} \Big(\mathcal{A}_{1i}(|\nabla u_i|^{p_i(x)}) + w_i(x)\mathcal{A}_{2i}(|u_i|^{p_i(x)}) \Big) dx \\ 
				&\quad - \lambda \int_{\mathbb{R}^N} \Big( \sum_{i=1}^n \int_0^{u_i} \frac{\partial F}{\partial s} (x, u_1, \dots, s, \dots, u_n) ds + F(x, 0, \dots, 0) dx \Big) \\
				 & \geq \sum_{i=1}^n \frac{m_0}{p_i^+} \int_{\mathbb R^N} \Big( \frac{1}{\beta_{1i}} a_{1i}\big( |\nabla u_i|^{p_i(x)} \big) |\nabla u_i|^{p_i(x)} + \frac{1}{\beta_{2i}}w_i(x) a_{2i}\big( |u_i|^{p_i(x)} \big) |u_i|^{p_i(x)} \Big) dx \\
				  & \quad- \lambda c_{3} \Big[\sum_{i=1}^{n} \Big(\sum_{j=1}^{n}|b_{i j}|_{\alpha_{i j}(x)}\|u_{j}\|_{\gamma_{j}(x)}^{\mu_{i j}-1}\|u_{i}\|_{\gamma_{i}(x)}\Big)\Big] \\ 
				  & \geq \sum_{i=1}^n  \frac{m_0}{p_i^+\max\{\beta_{1i}, \beta_{2i}\}} \int_{\mathbb R^N} \Big(a_{1i}\big( |\nabla u_i|^{p_i(x)} \big) |\nabla u_i|^{p_i(x)} + w_i(x) a_{2i}\big(|u_i|^{p_i(x)} \big) |u_i|^{p_i(x)} \Big) dx \\ 
				  & \quad - \lambda c_{3} \Big[\sum_{i=1}^{n} \Big(\sum_{j=1}^{n}|b_{i j}|_{\alpha_{i j}(x)}\|u_{j}\|_{\gamma_{j}(x)}^{\mu_{i j}-1}\|u_{i}\|_{\gamma_{i}(x)}\Big)\Big] \\ 
				  & \geq \sum_{i=1}^n \frac{m_0\min\big\{ \min \{k_{1i}^0 , k_{2i}^0\} , \min \{k_{1i}^2 , k_{2i}^2\} \big\}}{p_i^+\max\{\beta_{1i}, \beta_{2i}\}}  \int_{\mathbb R^N}\Big[ \Big(| \nabla u_i|^{p_i(x)}+w_i(x) |u_i|^{p_i(x)} \Big)  \\ 
				  & \quad   + \mathcal{H}(k_{i}^3) \Big(| \nabla u_i|^{q_i(x)}  + w_i(x) |u_i|^{q_i(x)} \Big)\Big]dx  - \lambda c_{3} \Big[\sum_{i=1}^{n} \Big(\sum_{j=1}^{n}|b_{i j}|_{\alpha_{i j}(x)}\|u_{j}\|_{\gamma_{j}(x)}^{\mu_{i j}-1}\|u_{i}\|_{\gamma_{i}(x)}\Big)\Big] \\ 
				& \geq \sum_{i=1}^n \frac{m_0 \kappa_i}{p_i^+} \cdot \Big(\rho_{p_i(x)w_i}(u_i) + \rho_{q_i(x)w_i}(u_i)\Big) - \lambda c_{3} \Big[\sum_{i=1}^{n} \Big(\sum_{j=1}^{n}|b_{i j}|_{\alpha_{i j}(x)}\|u_{j}\|_{\gamma_{j}(x)}^{\mu_{i j}-1}\|u_{i}\|_{\gamma_{i}(x)}\Big)\Big] \end{aligned}$
		\end{center}
		
		By using Young's inequality and \eqref{ineq1}, we obtain
		
		\begin{center}
			$\displaystyle \begin{aligned} E_\lambda (u) & \geq \sum_{i=1}^n \frac{m_0 \kappa_i}{p_i^+}  \Big(\|u\|_{p_i(x),w_i}^{p_i^-} + \mathcal{H}(k_{i}^3) \|u\|_{q_i(x),w_i}^{q_i^-} \Big)  - \lambda c_{3} \left[ \sum_{i=1}^{n} \left( \sum_{j=1}^{n} \left|b_{i j}\right|_{\alpha_{i j}(x)} \left(\frac{\mu_{i j}-1}{\mu_{i j}}\left\|u_{j}\right\|_{\gamma_{j}(x)}^{\mu_{i j}}\right. \right. \right.\left.\left.\left. +\frac{1}{\mu_{i j}}\left\|u_{i}\right\|_{\gamma_{i}(x)}^{\mu_{i j}}\right)\right) \right] \\
				 &  \geq \sum_{i=1}^n \left[ \frac{m_0 \kappa_i}{p_i^+}  \Big(\|u\|_{p_i(x),w_i}^{p_i^-} + \mathcal{H}(k_{i}^3) \|u\|_{q_i(x),w_i}^{q_i^-} \Big)  - \lambda c_{4}\left(\sum_{j=1}^{n}\left(\left|b_{i j}\right|_{\alpha_{i j}(x)}\left\|u_{j}\right\|_{\gamma_{j}(x)}^{\mu_{i j}}\right.\right.\right. \left.\left.\left.+\left|b_{i j}\right|_{\alpha_{i j}(x)}\left\|u_{i}\right\|_{\gamma_{i}(x)}^{\mu_{i j}}\right)\right)\right]. \end{aligned} $
		\end{center}
		
		This shows that $(\Phi - \lambda \Psi)(u) \longrightarrow +\infty$ as $\|u\| \longrightarrow +\infty$, since we have $\displaystyle 1 < \mu_{ij} < \inf_{x \in \mathbb{R}^N} \gamma_i(x)$. Therefore, it is confirmed that $\Phi - \lambda \Psi$ is coercive on $X$, for every parameter $\lambda>0$ in particular for every $\displaystyle \lambda \in \Lambda_r = ]\frac{\Phi(\bar x)}{\Psi(\bar x)}, \frac{r}{\displaystyle \sup_{\Phi(\bar x) \leq r} \Psi(\bar x)}[$,
		and then the hypothesis $(a_2)$ of  Lemma \ref{lem1.1}  is also proved. Now, all the hypotheses of Lemma \ref{lem1.1} are satisfied. Note that the solutions of the equation $\Phi^\prime (u) - \lambda \Psi^\prime (u)$ are exactly the weak solutions of \eqref{s1.1}. Thus, for each
		
		\begin{center}
			$\displaystyle \lambda \in \Bigg] \frac{\displaystyle m_1 \sum_{i=1}^n\Big( \max \left\lbrace \| z_i\|_{p_i(x),w_i}^{p_i^-} ,\| z_i\|_{p_i(x),w_i}^{p_i^+} \right\rbrace + \mathcal{H}(k_i^3) \max \left\lbrace \| z_i\|_{q_i(x),w_i}^{q_i^-} ,\| z_i\|_{q_i(x),w_i}^{q_i^+} \right\rbrace\Big)}{\displaystyle \int_{\mathbb{R}^N} F(x, z_1, \dots, z_n) dx} , \frac{r}{\displaystyle \int_{\mathbb{R}^N} \sup_{(\xi_1, \dots, \xi_n) \in K(\frac{sr}{m_0^{\star}})} F(x, \xi_1, \dots, \xi_n) dx} \Bigg[$
		\end{center}
		
		system \eqref{s1.1} admits at least three distinct weak solutions in $X$.
	\end{proof}




\end{document}